\documentclass[]{article}

\usepackage[utf8]{inputenc}
\usepackage{amsmath,amsfonts,amssymb,amsthm}
	\numberwithin{equation}{section} 
\usepackage{xcolor}
\usepackage{enumerate}
\usepackage{geometry}
\usepackage[unicode]{hyperref}   

\newtheorem{theorem}{\protect\theoremname}[section]
\newtheorem{definition}[theorem]{\protect\definitionname}
\newtheorem{lemma}[theorem]{\protect\lemmaname}
\newtheorem{proposition}[theorem]{\protect\propositionname}
\newtheorem{example}[theorem]{\protect\examplename}
\newtheorem{corollary}[theorem]{\protect\corollaryname}
\newtheorem{remark}[theorem]{\protect\remarkname}
\newtheorem{problem}[theorem]{\protect\problemname}
\newtheorem{notation}[theorem]{\protect\notationname}

\providecommand{\corollaryname}{Corollary}
\providecommand{\claimname}{Claim}
\providecommand{\definitionname}{Definition}
\providecommand{\lemmaname}{Lemma}
\providecommand{\notationname}{Notation}
\providecommand{\remarkname}{Remark}
\providecommand{\problemname}{Problem}
\providecommand{\propositionname}{Proposition}
\providecommand{\examplename}{Example}
\providecommand{\theoremname}{Theorem}
\providecommand{\conjecturename}{Conjecture}

\newcommand{\N}{\mathbb{N}}

\newcommand{\slantfrac}[2]{\,^#1\!/_#2}
\newcommand{\mc}{\mathcal}

\newcommand{\fsd}{F_{\sigma\delta}}

\newcommand{\ext}{\hat{\ }}

\newcommand{\baire}{\omega^\omega}

\newcounter{vkNoteCounter}

\newcounter{okNoteCounter}

\begin{document}

\begin{titlepage}
\title{Absolute $\mc F$-Borel classes}
\author{Vojtěch Kovařík\\
	E-mail: kovarikv@karlin.mff.cuni.cz 	\\
	\\
	Charles University \\
	Faculty of Mathematics and Physics \\
	Department of Mathematical Analysis \\
	\\
	Sokolovská 83, Karlín \\
	Praha 8, 186 00 \\
	Czech Republic}
	
\date{}

\maketitle

\renewcommand{\thefootnote}{}

\footnote{2010 \emph{Mathematics Subject Classification}: Primary 54H05; Secondary 54G20.}

\footnote{\emph{Key words and phrases}: compactification, Borel set, F-Borel class, absolute complexity.}

\renewcommand{\thefootnote}{\arabic{footnote}}
\setcounter{footnote}{0}

\begin{abstract}
We investigate and compare $\mc F$-Borel classes and absolute $\mc F$-Borel classes. We provide precise examples distinguishing these two hierarchies. We also show that for separable metrizable spaces, $\mc F$-Borel classes are automatically absolute. 
\end{abstract}

\end{titlepage}

\setcounter{page}{2}

\section{Introduction}\label{section: introduction}
\emph{Borel} sets (in some topological space $X$) are the smallest family containing the open sets, which is closed under the operations of taking countable intersections, countable unions and complements (in $X$). The family of \emph{$\mc F$-Borel} (resp. \emph{$\mc G$-Borel}) sets is the smallest family containing all closed (resp. open) sets, which is closed under the operations of taking countable intersections and countable unions of its elements. In metrizable spaces the families of Borel, $\mc F$-Borel and $\mc G$-Borel sets coincide. In non-metrizable spaces, open set might not necessarily be a countable union of closed sets, so we need to make a distinction between Borel, $\mc F$-Borel and $\mc G$-Borel sets.

In the present paper, we investigate absolute $\mc F$-Borel classes. While Borel classes are absolute by \cite{holicky2003perfect} (see Proposition \ref{proposition: HS} below), by \cite{talagrand1985choquet} it is not the case for $\mc F$-Borel classes (see Theorem \ref{theorem: talagrand} below). We develop a method of estimating absolute Borel classes from above, which enables us to compute the exact complexity of Talagrand's examples and to provide further examples by modifying them. Our main results are Theorem \ref{theorem: main theorem} and Corollary \ref{corollary: non K alpha space which is abs K alpha+1}.

The paper is organized as follows: In the rest of the introductory section, we define the Borel, $\mc F$-Borel and $\mc G$-Borel hierarchies and recall some basic results. In Section \ref{section: compactifications}, we recall the definitions and basic results concerning
compactifications and their ordering. In Theorem \ref{theorem: separable metrizable}, we
show that the complexity of separable metrizable spaces coincides with their absolute complexity (in any of the hierarchies). Section \ref{section: subspaces} is devoted to showing that absolute complexity is inherited by closed subsets. In Section \ref{section: sequences and brooms}, we study special sets of sequences of integers -- trees, sets which extend to closed discrete subsets of $\baire$ and the `broom sets' introduced by Talagrand. We study in detail the hierarchy of these sets using the notion of rank. In Section \ref{section: broom spaces and parametrization}, we introduce the class of examples of spaces used by Talagrand. We investigate in detail their absolute complexity and in Section \ref{section: absolute complexity of brooms}, we prove our main results.

Let us start by defining the basic notions. Throughout the paper, all the spaces will be Tychonoff. For a family of sets $\mc C$, we will denote by $\mc C_\sigma$ the collection of all countable unions of elements of $\mc C$ and by $\mc C_\delta$ the collection of all countable intersections of elements of $\mc C$.
\begin{definition}[Borel classes]\label{definition: borel classes}
Let $X$ be a topological space. We define the Borel multiplicative classes $\mc M_\alpha(X)$ and Borel additive classes $\mc A_\alpha(X)$, $\alpha<\omega_1$, as follows:
\begin{itemize}
\item $\mc M_0(X)=\mc A_0(X):=$ the algebra generated by open subsets of $X$,
\item $\mc M_\alpha(X):=\left( \underset {\beta<\alpha} \bigcup \left( \mc M_\beta(X) \cup \mc A_\beta(X) \right) \right)_\delta$ for $1\leq \alpha <\omega_1$,
\item $\mc A_\alpha(X):=\left( \underset {\beta<\alpha} \bigcup \left( \mc M_\beta(X) \cup \mc A_\beta(X) \right) \right)_\sigma$ for $1\leq \alpha <\omega_1$.
\end{itemize}
\end{definition}
In any topological space $X$, we have the families of closed sets and open sets, denoted by $F(X)$ and $G(X)$, and we can continue with classes $F_\sigma(X)$, $G_\delta(X)$, $F_{\sigma\delta}(X)$ and so on. However, this notation quickly gets impractical, so we use the following notation.
\begin{definition}[$\mc F$-Borel and $\mc G$-Borel classes]
We define the hierarchy of $\mc F$-Borel sets on a topological space $X$ as follows:
\begin{itemize}
\item $\mc F_1(X):=$ closed subsets of $X$,
\item $\mc F_\alpha(X):=\left( \underset {\beta<\alpha} \bigcup \mc F_\beta(X) \right)_\sigma$ for $2\leq \alpha <\omega_1$ even,
\item $\mc F_\alpha:=\left( \underset {\beta<\alpha} \bigcup \mc F_\beta (X)\right)_\delta$ for $3\leq \alpha <\omega_1$ odd.
\end{itemize}
The sets of $\alpha$-th $\mc G$-Borel class, $\mc G_\alpha(X)$, are the complements of $\mc F_\alpha(X)$ sets.
\end{definition}
By a \emph{descriptive class of sets}, we will always understand one of the Borel, $\mc F$-Borel or $\mc G$-Borel classes.

\begin{remark}\label{remark: F,G and M,A}
$(i)$ In any topological space $X$, we have
\[ \mc F_1(X) = F(X) \subset \mc M_0(X). \]
It follows that each class $\mc F_n$, $n\in\N$, is contained in $\mc A_{n-1}$ or $\mc M_{n-1}$ (depending on the parity). The same holds for classes $\mc F_\alpha$, $\alpha \geq \omega$, except that the difference between ranks disappears.

$(ii)$ If $X$ is metrizable, the Borel, $\mc F$-Borel and $\mc G$-Borel classes are related to the standard Borel hierarchy \cite[ch.11]{kechris2012classical}. In particular, each $\mc F_\alpha(X)$ is equal to either $\Sigma^0_\alpha(X)$ or $\Pi^0_\alpha(X)$ (depending on the parity of $\alpha$) and we have
\[ F(X)\cup G(X) \subset \mc M_0(X)=\mc A_0(X) \subset F_\sigma(X)\cap G_\delta(X). \]
The relations between our descriptive classes can then be summarized as:
\begin{itemize}
\item $\mc F_n(X)=\Pi^0_n(X)=\mc M_{n-1}(X)$ holds if $1\leq n<\omega$ is odd,
\item $\mc F_n(X)=\Sigma^0_n(X)=\mc A_{n-1}(X)$ holds if $2\leq n<\omega$ is even,
\item $\mc F_\alpha(X)=\Sigma^0_\alpha(X)=\mc A_\alpha(X)$ holds if $\omega\leq \alpha <\omega_1$ is even,
\item $\mc F_\alpha(X)=\Pi^0_\alpha(X)=\mc M_\alpha(X)$ holds if $\omega\leq \alpha <\omega_1$ is odd.
\end{itemize}
Clearly, the $\mc G$-Borel classes satisfy the dual version of $(i)$ and $(ii)$.
\end{remark}
Note also that in compact spaces, closed sets are compact, so in this context the $\mc F$-Borel sets are sometimes called $\mc K$-Borel, $F_\sigma$ sets are called $K_\sigma$ and so on.

Let us define two notions central to the topic of our paper.
\begin{definition}\label{definition: F sigma delta spaces}
Let $X$ be a (Tychonoff) topological space and $\mc C$ be a descriptive class of sets. We say that $X$ is a \emph{$\mc C$ space} if there exists a compactification $cX$ of $X$, such that $X\in\mc C(cX)$.

If $X\in\mc C(Y)$ holds for any Tychonoff topological space $Y$ in which $X$ is densely embedded, we say that $X$ is an \emph{absolute $\mc C$ space} We call the class $\mc C$ absolute if every $\mc C$ space is an absolute $\mc C$ space.
\end{definition}

The basic relation between complexity and absolute complexity are noted in the following remark.
\begin{remark}\label{remark: 1-5}
Consider the following statements:
\begin{enumerate}[(i)]
\item $X\in\mc C(cX)$ holds for some compactification $cX$;
\item $X\in\mc C(\beta X)$ holds for the Čech-Stone compactification;
\item $X\in\mc C(cX)$ holds for every compactification $cX$;
\item $X\in\mc C(Y)$ holds for every Tychonoff space where $Y$ is densely embedded;
\item $X\in\mc C(Y)$ holds for every Tychonoff space where $Y$ is embedded.
\end{enumerate}
Clearly, the implications $(v)\implies(iv)\implies(iii)\implies(ii)\implies(i)$ always hold. In the opposite direction, we always have $(i)\implies(ii)$ (this is standard, see Remark \ref{remark: absolute complexity}) and $(iii)\implies (iv)$. For Borel and $\mc F$-Borel classes, $(iv)$ is equivalent to $(v)$ (since these classes are closed under taking intersections with closed sets). For $\mc G$-Borel classes, $(iv)$ is never equivalent to $(v)$ (just take $Y$ so large that $\overline{X}^Y$ is not $G_\delta$ in $Y$).
\end{remark}
The interesting part is therefore the relation between $(i)$ and $(iii)$ from Remark \ref{remark: 1-5}. For the first two levels of $\mc F$-Borel and $\mc G$-Borel hierarchies, $(i)$ is equivalent to $(iii)$, that is, these classes are absolute. A more precise formulation is given in the following remark.

\begin{remark}\label{remark: F,G,Fs,Gd are absolute}
For a topological space $X$ we have
\begin{enumerate}[(i)]
	\item $X$ is an $\mc F_1$ space $\iff$ $X$ is absolutely $\mc F_1$ $\iff$ $X$ is compact;
	\item $X$ is an $\mc F_2$ space $\iff$ $X$ is absolutely $\mc F_2$ $\iff$ $X$ is $\sigma$-compact;
	\item $X$ is a $\mc G_1$ space $\iff$ $X$ is absolutely $\mc G_1$ $\iff$ $X$ is locally compact;
	\item $X$ is a $\mc G_2$ space $\iff$ $X$ is absolutely $\mc G_2$ $\iff$ $X$ is Čech-complete
\end{enumerate}
(for the proofs of $(iii)$ and $(iv)$ see \cite[Theorem 3.5.8 and Theorem 3.9.1]{engelking1989general}).
\end{remark}

The first counterexample (and the only one known to the author) was found by Talagrand, who showed that already the class of $\fsd$ sets is not absolute:
\begin{theorem}[\cite{talagrand1985choquet}]\label{theorem: talagrand}
There exists an $\mc F_3$ space $T$ and its compactification $K$, such that $T$ is not $\mc F$-Borel in $K$.
\end{theorem}
This not only shows that none of the classes $\mc F_\alpha$ for $\alpha\geq 4$ are absolute, but also that the difference between absolute and non-absolute complexity can be large. Indeed, the `non-absolute complexity' of $T$ is `$\mc F_3$', but its `absolute complexity' is `$\mc K$-analytic'. In \cite{kalenda2017absolute}, the authors give a sufficient condition for an $\mc F_3$ space to be absolutely $\mc F_3$, but a characterization of absolutely $\mc F_\alpha$ spaces is still unknown for all $\alpha\geq 3$.

In \cite{holicky2003perfect} and \cite{raja2002some}, the authors studied (among other things) absoluteness of Borel classes. In particular, the following result is relevant to our topic:
\begin{proposition}[{\cite[Corollary 14]{holicky2003perfect}}]\label{proposition: HS}
For every $1\leq \alpha < \omega_1$, the classes $\mc A_\alpha$ and $\mc M_\alpha$ are absolute.
\end{proposition}
Another approach is followed in \cite{marciszewski1997absolute} and \cite{junnila1998characterizations}, where the absoluteness is investigated in the metric setting (that is, for spaces which are of some class $\mc C$ when embedded into any \emph{metric} space). In \cite{junnila1998characterizations} a characterization of `metric-absolute' $\mc F_3$ spaces is given in terms of existence complete sequences of covers (a classical notion used by Frolík for characterization of Čech-complete spaces, see \cite{frolik1960generalizations}) which are $\sigma$-discrete. Unfortunatelly, this is not applicable to the topological version of absoluteness, because every countable cover is (trivially) $\sigma$-discrete, and any $\mc F_3$ space does have such a cover by \cite{frolik1963descriptive} -- even the Talagrand's non-absolute space.

\section{Compactifications and their ordering}\label{section: compactifications}
By a \emph{compactification} of a topological space $X$ we understand a pair $(cX,\varphi)$, where $cX$ is a compact space and $\varphi$ is a homeomorphic embedding of $X$ onto a dense subspace  of $cX$. Symbols $cX$, $dX$ and so on will always denote compactifications of $X$.

Compactification $(cX,\varphi)$ is said to be \emph{larger} than $(dX,\psi)$, if there exists a continuous mapping $f : cX\rightarrow dX$, such that $\psi = f \circ \varphi$. We denote this as $ cX \succeq dX $. Recall that for a given $T_{3\slantfrac{1}{2}}$ topological space $X$, its compactifications are partially ordered by $\succeq$ and Stone-Čech compactification $\beta X$ is the largest one.

Often, we encounter a situation where $X\subset cX$ and the corresponding embedding is identity. In this case, we will simply write $cX$ instead of $(cX,\textrm{id}|_X)$.

Much more about this topic can be found in many books, see for example \cite{freiwald2014introduction}. The basic relation between the complexity of a space $X$ and the ordering of compactifications is the following:
\begin{remark}\label{remark: absolute complexity}
If $\mc C$ is a descriptive classes of sets, we have
 \[ X\in\mc C(dX), \ cX\succeq dX \implies X\in\mc C(cX). \]
In particular, $X$ is a $\mc C$ space if and only if $X\in\mc C(\beta X)$.
\end{remark}
We will also make use of the following result about existence of small metrizable compactifications.
\begin{proposition}\label{proposition: metrizable sub-compactification}
Let $X$ be a separable metrizable space and $cX$ its compactification. Then $X$ has some metrizable compactification $dX$, such that $dX\preceq cX$.
\end{proposition}
This proposition is an exercise, so we only include a sketch of its proof:
\begin{proof}
We can assume that $cX\subset [0,1]^\kappa$ for some $\kappa$. Since $X$ has a countable base, there is a countable set of coordinates $I\subset \kappa$ such that the family of all $V \cap X$,
where $V$ is a basic open set $V\subset [0,1]^\kappa$ depending on coordinates in $I$ only is a base
of $X$. Then we can take $dX := \pi_I (cX)$, where $\pi_I$ denotes the projection.
\end{proof}

We put together several known results to obtain the following theorem (which we have not found anywhere in the literature):
\begin{theorem}\label{theorem: separable metrizable}
Let $\mc C$ be a descriptive class of sets. Then any separable metrizable $\mc C$ space is an absolute $\mc C$ space (so the first four conditions of Remark \ref{remark: 1-5} are equivalent).

If $\mc C$ is one of the Borel or $\mc F$-Borel classes, this is further equivalent to $X$ being in $\mc C(Y)$ for every Tychonoff space in which $Y$ is embedded.
\end{theorem}
\begin{proof}
The statement for the classes $\mc A_\alpha$ and $\mc M_\alpha$ follows from the general result of \cite{holicky2003perfect} (see Proposition \ref{proposition: HS} above). Let us continue by proving the result for the $\mc F$-Borel classes. For the first two levels, it follows from Remark \ref{remark: F,G,Fs,Gd are absolute}.
Suppose that $X\in \mc F_\alpha(\beta(X))$ for some $\alpha\ge 3$ and let $cX$ be any compactification of $X$.

By Proposition \ref{proposition: metrizable sub-compactification} we can find a smaller metrizable compactification $dX$. Let $\mc D$ be the Borel class corresponding to $\mc F_\alpha$ (see Remark \ref{remark: F,G and M,A}). Then $X\in \mc D(\beta X)$. Since $\mc D$ is absolute by Proposition \ref{proposition: HS}, we get $X\in\mc D(dX)$. Since $dX$ is metrizable, we deduce that $X\in\mc F_\alpha(dX)$, so $X\in \mc F_\alpha(cX)$ holds by Remark \ref{remark: absolute complexity}.

The proof for $\mc G$-Borel classes is analogous.

\end{proof}

\section{Hereditarity of absolute complexity} \label{section: subspaces}
In this section, we show that absolute complexity is hereditary with respect to closed subspaces. To do this, we first need to be able to extend compactifications of subspaces to compactifications of the original space. We start with a topological lemma:

\begin{lemma}\label{lemma: quotient of K}
Suppose that $K$ is compact, $F\subset K$ is closed and $f:F\rightarrow L$ is a continuous surjective mapping. Define a mapping $q: K \to L\cup (K\setminus F)$ as $q(x)=x$ for $x\in K\setminus F$ and $q(x)=f(x)$ for $x\in F$. Then $\left(K\setminus F\right)\cup L$ equipped with the quotient topology (induced by $q$) is a compact space.
\end{lemma}
\begin{proof}
Since $f$ is surjective, the space $(K\setminus F)\cup L$ coincides with the adjunction space $K\cup_f L$ determined by $K$, $L$ and $f$ (for definition, see \cite[below Ex. 2.4.12]{engelking1989general}). Since $K\cup_f L$ is a quotient of a compact space $K\oplus L$ with respect to a closed equivalence relation (\cite[below Ex. 2.4.12]{engelking1989general}), we get by Alexandroff theorem (\cite[Thm 3.2.11]{engelking1989general}) that the space $K\cup_f L$ is compact (in particular, it is Hausdorff).
\end{proof}
This gives us a way to extend a compactification, provided that we already know how to extend some bigger compactification:
\begin{proposition}\label{proposition: extending a compactification}
Let $Y$ be a topological space and suppose that for a closed $X\subset Y$ we have $cX\preceq \overline{X}^{dY}$ for some compactifications $cX$ and $dY$. Then $cX$ is equivalent to $\overline{X}^{cX}$ for some compactification $cY$ of $Y$.
\end{proposition}
\begin{proof}
Let $X$, $Y$, $cX$ and $dY$ be as above and denote by $f:\overline{X}^{dY}\rightarrow cX$ the mapping which witnesses that $cX$ is smaller than $\overline{X}^{dY}$. We will assume that $X\subset cX$ and $Y\subset dY$, which means that $f|_X=\textrm{id}_X$.

By Lemma \ref{lemma: quotient of K} (with $K:=dY$, $F:=\overline{X}^{dY}$ and $L:=cX$), the space $cY$, defined by the formula
\[ cY:= \left(dY\setminus \overline{X}^{dY} \right)\cup cX =  \left(K\setminus F\right)\cup L, \]
is compact. To show that $cY$ is a compactification of $Y$, we need to prove that the mapping $q: dY\rightarrow cY$, defined as $f$ extended to $dY\setminus \overline{X}^{dY}$ by identity, is a homeomorphism when restricted to $Y$. The continuity of $q|_Y$ follows from the continuity of $q$. The restriction is injective, because $\psi$ is injective on $\left( dY \setminus \overline{X}^{dY}\right) \cup X$ (and $X$ is closed in $Y$). The definition of $q$ and the fact that $X$ is closed in $Y$ imply that the inverse mapping is continuous on $Y\setminus X$. To get the continuity of $(q|_Y)^{-1}$ at the points of $X$, let $x\in X$ and $U$ be a neighborhood of $x$ in $K$. We want to prove that $q(U\cap Y)$ is a neighborhood of $q(x)$ in $q(Y)$. To see this, observe that $C:=dY\setminus U$ is compact, therefore $q(C)$ is closed, its complement $V:=cY \setminus q(C)$ is open and it satisfies $V\cap q(Y) = q(U\cap Y)$.

Lastly, we note that for $y\in cY\setminus \overline{q(X)}^{cY}$, $q^{-1}(y)$ is in $dY\setminus \overline{X}^{dY}$. This gives the second identity on the following line, and completes the proof:
\[ \overline{X}^{cY} = \overline{q(X)}^{cY} = q\left( \overline{X}^{dY} \right) = f\left( \overline{X}^{dY} \right)=cX. \]
\end{proof}

Recall that a subspace $X$ of $Y$ is said to be \emph{$C^\star$-embedded} in $Y$ if any bounded continuous real function on $X$ can be continuously extended to $Y$.
\begin{corollary}\label{corollary: extending all compactifications}
If a closed set $X$ is $C^\star$-embedded in $Y$, then each compactification $cX$ is of the form $cX=\overline{X}^{cY}$ for some compactification $cY$.
\end{corollary}
\begin{proof}
Recall that the Čech-Stone compactification $\beta X$ is characterized by the property that each bounded continuous function from $X$ has a unique continuous extension to $\beta X$. Using this characterization, it is a standard exercise to show that $X$ is $C^\star$-embedded in $Y$ if and only if $\overline{X}^{\beta Y}=\beta X$. It follows that for $X$ and $Y$ as above, any compactification $cX$ is smaller than $\overline{X}^{\beta Y}$, and the result follows from Proposition \ref{proposition: extending a compactification}.
\end{proof}
We use these results to get the following corollary:
\begin{proposition}\label{proposition: absoluteness is hereditary}
Let $\mc C$ be one of the classes $\mc F_\alpha$, $\mc M_\alpha$ or $\mc A_\alpha$ for some $\alpha<\omega_1$.
\begin{enumerate}[(i)]
\item Any closed subspace of a $\mc C$ space is a $\mc C$ space;
\item Any closed subspace of an absolute $\mc C$ space is an absolute $\mc C$ space.
\end{enumerate}
\end{proposition}
\begin{proof}
$(i)$: This is trivial, since if $X$ is a closed subspace of $Y$ and $Y$ is of the class $\mc C$ in some compactification $cY$, then $X$ is of the same class in the compactification $\overline{X}^{cY}$.

$(ii)$: For Borel classes, which are absolute, the result follows from the first part -- therefore, we only need to prove the statement for the $\mc F$-Borel classes. Let  $\alpha<\omega_1$ and assume that $X$ is a closed subspace of some $\mc F_\alpha$ space $Y$.

Firstly, if $Y$ is an $\mc F_\alpha$ subset of $\beta Y$, it is $\mc K$-analytic and, in particular, Lindelöf. Since we also assume that $Y$ is Hausdorff, we get that $Y$ is normal \cite[Propositions 3.4 and 3.3]{kkakol2011descriptive}. By Tietze's theorem, $X$ is $C^\star$-embedded in $Y$, which means that every compactification of $X$ can be extended to a compactification of $Y$ (Corollary \ref{corollary: extending all compactifications}). We conclude the proof as in $(i)$.
\end{proof}

Note that the second part of Proposition \ref{proposition: absoluteness is hereditary} does not hold, in a very strong sense, if we replace the closed subspaces by $\fsd$ subspaces. Indeed, the space $T$ from Theorem \ref{theorem: talagrand} is an $\fsd$ subspace of a compact space, but it is not even absolutely $\mc F$-Borel (that is, $\mc F$-Borel in every compactification). Whether anything positive can be said about $F_\sigma$ subspaces of absolute $\mc F_\alpha$ spaces is unknown.

\section{Special sets of sequences}\label{section: sequences and brooms}
In this chapter, we define `broom sets' -- a special type sets of finite sequences, which have a special `tree' structure and all of their infinite extensions are closed and discrete. As in \cite{talagrand1985choquet}, these sets will then be used to construct spaces with special complexity properties.

\subsection{Trees and a rank on them}
First, we introduce the (mostly standard) notation which will be used in the sequel.
\begin{notation}[Integer sequences of finite and infinite length]\label{notation: sequences}
We denote
\begin{itemize}
\item $\omega^\omega:=$ infinite sequences of non-negative integers $:=\left\{\sigma : \omega \rightarrow \omega\right\}$,
\item $\omega^{<\omega}:=$ finite sequences of non-neg. integers $:=\left\{s : n \rightarrow \omega |\ n\in\omega \right\}$.
\end{itemize}
Suppose that $s\in\omega^{<\omega}$ and $\sigma\in\baire$. We can represent $\sigma$ as $(\sigma(0),\sigma(1),\dots)$ and $s$ as $(s(0),s(1),\dots,s(n-1))$ for some $n\in\omega$. We denote the \emph{length} of $s$ as $|s|=\textrm{dom}(s)=n$, and set $|\sigma|=\omega$. If for some $t\in \omega^{<\omega} \cup \baire$ we have $|t|\geq |s|$ and $t|_{|s|}=s$, we say that $u$ \emph{extends} $s$, denoted as $t\sqsubset t$. We say that  $u,v\in\omega^{<\omega}$ are non-comparable, denoting as $u\perp v$, when neither $u\sqsubset v$ nor $u\sqsupset v$ holds.

Unless we say otherwise, $\omega^\omega$ will be endowed with the standard product topology, whose basis consists of sets $\mc N (s):=\left\{ \sigma \sqsupset s | \ \sigma\in\baire \right\}$, $s \in \omega^{<\omega}$. 

For $n\in\omega$, we denote by $(n)$ the corresponding sequence of length $1$. By $s\ext t$ we denote the concatenation of a sequences $s\in\omega^{<\omega}$ and $t\in\omega^{<\omega}\cup \baire$. We will also use this notation to extend finite sequences by integers and sets, using the convention $s \hat{\ } k:=s \hat{\ } (k)$ for $k\in\omega$ and $s \hat \ T:=\left\{s\hat\ t\ |\ t\in T\right\}$ for $T\subset \omega^{<\omega} \cup \baire$.
\end{notation}

\begin{definition}[Trees on $\omega$]
A \emph{tree} (on $\omega$) is a set $T\subset \omega^{<\omega}$ which satisfies
\[ (\forall s,t\in\omega^{<\omega}): s\sqsubset t \ \& \ t\in T \implies s\in T. \]
By $\textrm{Tr}$ we denote the space of all trees on $\omega$. For $S\subset \omega^{<\omega}$ we denote by $\mathrm{cl}_\mathrm{Tr}(S):=\left\{u\in\omega^{<\omega} | \ \exists s\in S: s\sqsupset u \right\}$ the smallest tree containing $S$. Recall that the empty sequence $\emptyset$ can be thought of as the `root' of each tree, since it is contained in any nonempty tree.

If each of the initial segments $\sigma|n$, $n\in\omega$, of some $\sigma\in\baire$ belongs to $T$, we say that $\sigma$ is an infinite branch of $T$. By $\textrm{WF}$ we denote the space of all trees which have no infinite branches (the `well founded' trees).
\end{definition}

We define a variant of the standard derivative on trees and the corresponding $\Pi^1_1$-rank. We do not actually use the fact that the rank introduced in Definition \ref{definition: derivative} is a $\Pi^1_1$ rank, but more details about such ranks and derivatives can be found in \cite[ch. 34 D,E]{kechris2012classical}.
\begin{definition}\label{definition: derivative}
Let $T\in \textrm{Tr}$. We define its `infinite branching' \emph{derivative} $D_i(T)$ as
\[ D_i(T):=\left\{ t\in T| \ t\sqsubset s \textrm{ holds for infinitely many }s\in T \right\}. \]
This operation can be iterated in the standard way:
\begin{align*}
D^0_i(T) & :=  T, \\
D^{\alpha +1}_i \left( T \right) & :=  D_i\left(D_i^\alpha\left(T\right)\right) \textrm{ for successor ordinals},\\
D_i^\lambda \left( T \right) & := \underset {\alpha < \lambda } \bigcap D_i^\alpha \left( T \right) \textrm{ for limit ordinals}.
\end{align*}
This allows us to define a rank $r_i$ on $\textrm{Tr}$ as $r_i(\emptyset):=-1$, $r_i(T):=0$ for finite trees and
\[ r_i(T):=\min \{ \alpha<\omega_1| \ D_i^{\alpha+1}\left( T \right) = \emptyset \}. \]
If no such $\alpha<\omega_1$ exists, we set $r_i(T):=\omega_1$.
\end{definition}

Note that on well founded trees, the derivative introduced in Definition \ref{definition: derivative} behaves the same way as the derivative from \cite[Exercise 21.24]{kechris2012classical}, but it leaves any infinite branches untouched. It follows from the definition that the rank $r_i$ of a tree $T$ is countable if and only if $T$ is well founded. We consider this approach better suited for our setting.

While this definition of derivative and rank is the most natural for trees, we can extend it to all subsets $S$ of $\omega^{<\omega}$ by setting $D_i(S):=D_i(\mathrm{cl}_\mathrm{Tr}(S))$ and defining the rest of the notions using this extended `derivative'. Clearly, if $S\in Tr$, this coincides with the original definition.

\begin{lemma}\label{lemma: rank and finite covers}
No set $S\subset \omega^{<\omega}$ which satisfies $r_i(S)<\omega_1$ can be covered by finitely many sets of rank strictly smaller than $r_i(S)$.
\end{lemma}
\begin{proof}
Note that if $T_1$ and $T_2$ are trees, $T_1\cup T_2$ is a tree as well, and we have $D_i^\alpha(T_1\cup T_2)=D_i^\alpha(T_1)\cup D_i^\alpha(T_2)$ for any $\alpha<\omega_1$. Moreover, for any $S_1,S_2\subset \omega^{<\omega}$ we have $\mathrm{cl}_\mathrm{Tr}(S_1\cup S_2)=\mathrm{cl}_\mathrm{Tr}(S_1)\cup \mathrm{cl}_\mathrm{Tr}(S_2)$. This yields the formula
\[ r_i(S_1\cup\dots\cup S_k)=\max \{ r_i(S_1),\dots,r_i(S_k) \} \]
for any $S_i\subset \omega^{<\omega}$, which gives the result.
\end{proof}

\subsection{Sets which extend into closed discrete subsets of \texorpdfstring{$\baire$}{the Baire space}}
\begin{definition}\label{definition: D}
A set $A\subset \baire$ is said to be an infinite extension of $B\subset \omega^{<\omega}$, if there exists a bijection $\varphi : B\rightarrow A$, such that $(\forall s\in B) : \varphi (s)\sqsupset s$. If $\tilde A$ has the same property as $A$, except that we have $\tilde A \subset \omega^{<\omega}$, it is said to be a (finite) extension of $B$.

By $\mc D\subset \mc P\left(\omega^{<\omega}\right)$ we denote the system of those sets $D\subset \omega^{<\omega}$ which satisfy
\begin{itemize}
\item[(i)] $(\forall s,t\in D)$ : $s\neq t \implies s \perp t$;
\item[(ii)] every infinite extension of $D$ is closed and discrete in $\baire$.
\end{itemize}  
\end{definition}

\begin{example}[Broom sets of the first class]\label{example: brooms}
Let $h\in \omega^{<\omega}$ and $s_n\in\omega^{<\omega}$ for $n\in\omega$ be finite sequences and suppose that $(f_n)_{n\in\omega}$ is an injective sequence of elements of $\omega$. Using these sequences, we define a `broom' set $B$ as
\[ B:=\left\{ h\hat{\ }f_n\hat{\ }s_n|\ n\in\omega\right\}. \]
The intuition is that $h$ is the `handle' of the broom, $s_n$ are the `bristles', and the sequence $(f_n)_n$ causes these bristles to `fork out' from the end of the handle. Note that the injectivity of $(f_n)_n$ guarantees that $B\in \mc D$ (see Proposition \ref{proposition: properties of D}).
\end{example}

\begin{proposition}[Properties of $\mc D$]\label{proposition: properties of D}
The system $\mc D$ has the following properties (where $D$ is an arbitrary element of $\mc D$):
\begin{enumerate}[(i)]
\item $\{ (n) | \ n\in\omega \} \in \mc D$;
\item $(\forall E\subset \omega^{<\omega}): E\subset D \implies E\in\mc D$;
\item $(\forall E\subset \omega^{<\omega}): E$ is a finite extension of $D \implies E\in\mc D$;
\item $\left(\forall h\in\omega^{<\omega}\right): h\ext D\in\mc D$;
\item For each sequence $\left(D_n\right)_n$ of elements of $\mc D$ and each one-to-one enumeration  $\left\{d_n| \ n\in\omega\right\}$ of $D$, we have $\bigcup _{n\in\omega} d_n\hat{\ }D_n \in \mc D$.
\end{enumerate}
\end{proposition}
\begin{proof}
The first four properties are obvious. To prove $5.$, let $D$, $D_n$, $E:=\bigcup _{n\in\omega} d_n\hat{\ }D_n$ be as above. The show that $E$ satisfies (i) from Definition \ref{definition: D}, suppose that $e=d_n\ext e',\ f=d_m\ext f' \in D$ are comparable. Then $d_n$ and $d_m$ are also comparable, which means that $m=n$ and $e',f'\in D_n=D_m$ are comparable. It follows that $e'=f'$ and hence $e=f$.

For $A\subset \omega^{<\omega}$, the condition (ii) from Definition \ref{definition: D} is equivalent to
\begin{equation}\label{eq: equivalence}
(\forall \sigma \in \baire)(\exists m\in\omega): \  \left\{ s\in A | \ s \textrm{ is comparable with } \sigma|m \right\} \textrm{ is finite}.
\end{equation}
Indeed, suppose \eqref{eq: equivalence} holds and let $H$ be an infinite extension of $A$. For $\sigma\in\baire$ let $m\in\omega$ be as in \eqref{eq: equivalence}. Any $\mu\in H\cap \mc N(\sigma|m)$ is an extension of some $s\in A$, which must be comparable with $\sigma|m$. It follows that the neighborhood $\mc N(\sigma|m)$ of $\sigma$ only contains finitely many elements of $H$, which means that $\sigma$ is not a cluster point of $H$ and $H$ is discrete.
In the opposite direction, suppose that the condition does not hold. Then we can find a one-to-one sequence $(s_m)_{m\in\omega}$ of elements of $A$, such that $s_m \sqsupset \sigma|m$ for each $m\in\omega$. Then $\sigma$ is a cluster point of (any) infinite extension of $\{ s_m|\ m\in\omega \}$.

If $A$ also satisfies (i) from Definition \ref{definition: D}, \eqref{eq: equivalence} is clearly equivalent to
\begin{equation}\label{eq: equivalence with (i)}
(\forall \sigma \in \baire)(\exists m\in\omega): \  \sigma|m \textrm{ is comparable with at most one } s\in A.
\end{equation}
To see that $E$ satisfies \eqref{eq: equivalence with (i)}, let $\sigma\in \baire$. Since $D\in \mc D$, let $m$ be the integer obtained by applying \eqref{eq: equivalence with (i)} to this set. If no $d\in D$ is comparable with $\sigma|m$, then neither is any $e\in E$ and we are done.

Otherwise, let $d_{n_0}$ be the only element of $D$ comparable with $\sigma|m$ and let  $m_0$ be the integer obtained by applying \eqref{eq: equivalence with (i)} to $d_{n_0}\ext D_{n_0}$. Then $m':=\max\{m_0,m\}$ witnesses that $E$ satisfies \eqref{eq: equivalence with (i)}: Indeed, any $e=d_n\ext e'$ comparable with $\sigma|{m'}$ must have $d_n$ comparable with $\sigma|{m'}$. In particular, $d_n$ is comparable with $\sigma|m$. It follows that $d_n=d_{n_0}$ and $e\in d_{n_0}\ext D_{n_0}$. Therefore, $e$ must be comparable with $\sigma|{m_0}$ and since there is only one such element, the proof is finished.
\end{proof}

\subsection{Broom sets}
We are now ready to define the broom sets and give some of their properties.
\begin{definition}[Broom sets of higher classes]\label{definition: brooms}
We define the hierarchy of broom sets $\mc B_\alpha$, $\alpha < \omega_1$. We set $\mc B_0:=\left\{ \left\{ s \right\} | \ s\in\omega^{<\omega} \right\}$ and, for $1<\alpha<\omega_1$ we define $\mc B_\alpha$ as the union of $\bigcup_{\beta < \alpha}\mc B_\beta$ and the collection of all sets $B\subset \omega^{<\omega}$ of the form
\[ B= \underset {n\in \omega}\bigcup h\hat{\ }f_n\hat{\ }B_n \]
for some finite sequence $h\in\omega^{<\omega}$ (`handle'), broom sets $B_n\in \bigcup_{\beta < \alpha}\mc B_\beta$ (`bristles') and one-to-one sequence $(f_n)_{n\in\omega} \in \baire$ (`forking sequence'). We also denote as $\mc B_{\omega_1}:=\underset {\alpha<\omega_1} \bigcup\mc B_\alpha$ the collection of all these broom classes.
\end{definition}

In Section \ref{section: broom spaces and parametrization}, we will need a result of the type `if $B$ is an element of $\mc B_{\alpha+1}\setminus \mc B_\alpha$, then $B$ cannot be covered by finitely many brooms of class $\alpha$'. While this is, in general, not true, we can restrict ourselves to the following subcollection $\widetilde{\mc B}_{\alpha+1}$ of $\mc B_{\alpha+1}\setminus \mc B_\alpha$ of `$(\alpha+1)$-brooms whose bristles are of the highest allowed class', where the result holds (as we will see later):
\begin{definition} \label{definition: tilde brooms}
We set $\widetilde {\mc B}_0:=\mc B_0$. For successor ordinals we define $\widetilde {\mc B}_{\alpha +1}$ as the collection of all sets $B\subset \omega^{<\omega}$ of the form
\[ B= \underset {n\in \omega}\bigcup h\hat{\ }f_n\hat{\ }B_n, \]
where $h\in\omega^{<\omega}$, $B_n\in \widetilde{\mc B}_\alpha$ and $(f_n)_n$ is the (1-to-1) forking sequence. For a limit ordinal $\alpha<\omega_1$, the family $\widetilde {\mc B}_{\alpha}$ is defined in the same way, except that the bristles $B_n$ belong to $\widetilde {\mc B}_{\beta_n}$ for some $\beta_n<\alpha$, where $\sup_n \beta_n=\alpha$.
\end{definition}

\begin{remark}[Basic properties of broom sets]\label{remark: basic broom properties}
It is clear that for every broom $B\in\mc B_\alpha$, any finite extension of $B$ is also a broom of class $\alpha$, and so is the set $h\ext B$ for any $h\in\omega^{<\omega}$. Moreover, it follows from Proposition \ref{proposition: properties of D} that $\mc B_{\omega_1} \subset \mc D$. These properties also hold for collections $\widetilde {\mc B}_\alpha$, $\alpha<\omega_1$.

Actually, broom sets from Definition \ref{definition: brooms} are exactly those elements of $\mc D$ for which every `branching point' is a `point of infinite branching'. However, we do not need this description of broom sets, so we will not give its proof here.
\end{remark}

Broom sets are connected with rank $r_i$ in the following way:
\begin{proposition}\label{proposition: rank of B}
For $B\in\mc B_{\omega_1}$ and $\alpha<\omega_1$, the following holds:
\begin{enumerate}[(i)]
\item $B\in \mc B_\alpha \implies r_i(B)\leq \alpha$,
\item $B\in \widetilde{\mc B}_\alpha \implies r_i(B) = \alpha$.
\end{enumerate}
\end{proposition}
\begin{proof}
Firstly, we make the following two observations:
\begin{equation*}
\begin{split}
(\forall B\subset \omega^{<\omega}): r_i(B)=-1 & \iff B=\emptyset, \\
(\forall B \in\mc B_{\omega_1}): r_i(B)=0 & \iff B \in\mc B_0 = \widetilde {\mc B}_0. 
\end{split}
\end{equation*}
We use this as a starting point for transfinite induction. We now prove $(i)$ and $(ii)$ separately, distinguishing also between the cases of successor $\alpha$ and limit $\alpha$.

$(i)$: Suppose that the statement holds for $\alpha<\omega_1$ and let $B\in {\mc B}_{\alpha+1}$. By definition, $B=\bigcup_n h\hat{\ }f_n\hat{\ }B_n$ holds for some handle $h\in\omega^{<\omega}$, forking sequence $(f_n)_n$ and $B_n\in\mc B_\alpha$. For each $n\in\omega$ we have $r_i(B_n)\leq \alpha$, which gives $D_i^{\alpha+1}(B_n)=\emptyset$ and hence $D_i^{\alpha+1}(f_n\hat{\ }B_n)=\emptyset$. It follows that
\[ D_i^{\alpha+1}(\bigcup_n h\hat{\ }f_n\hat{\ }B_n) \subset \mathrm{cl}_\mathrm{Tr}(h), \]
and
\[ D_i^{\alpha+2}(\bigcup_n h\hat{\ }f_n\hat{\ }B_n) = \emptyset ,\]
which, by definition or $r_i$, means that $r_i(B)\leq \alpha+1$.

For the case of limit ordinal $\alpha<\omega_1$, assume that the statement holds for every $\beta<\alpha$, and let $B\in {\mc B}_{\alpha}$. As above, $B=\bigcup_n h\hat{\ }f_n\hat{\ }B_n$ holds for some $h$, $(f_n)_n$ and $B_n$, where $B_n\in\mc B_{\beta_n}$, $\beta_n + 1 <\alpha$. We successively deduce that $r_i(B_n)$ is at most $\beta_n$, hence $D_i^{\alpha}(B_n) \subset D_i^{\beta_n+1}(B_n)$ holds and we have $D_i^{\alpha}(f_n\hat{\ }B_n)=\emptyset$.
Since $(f_n)_n$ is a forking sequence, then for any $s\in f_{n_0}\hat{\ }B_{n_0}$ which is different from the empty sequence, we get the following equivalence:
\begin{equation}\label{eq: r(B)<=alpha}
h\hat{\ }s\in D_i^{\alpha}(\bigcup_n h\hat{\ }f_n\hat{\ }B_n) \iff
 \left( \exists n_0\in \omega \right): s\in D_i^{\alpha}(f_{n_0}\hat{\ }B_{n_0}).
\end{equation}
From \eqref{eq: r(B)<=alpha} we conclude that $D_i^{\alpha}(B)\subset \mathrm{cl}_\mathrm{Tr}(\{h\})$ is finite, which means that $r_i(B)\leq \alpha$.

$(ii)$: Suppose that the second part of the proposition holds for some $\alpha<\omega$ and let $B=\bigcup_n h\hat{\ }f_n\hat{\ }B_n\in\widetilde{\mc B}_{\alpha+1}$. For each $n\in\omega$, $B_n$ is an element of $\widetilde{\mc B}_\alpha$, which by the induction hypothesis means that $r_i(B_n) = \alpha$, and thus there exists some sequence $s_n\in D_i^\alpha(B)$. In particular, we have $h\hat{\ }f_n\hat{\ }s_n\in D_i^\alpha(h\hat{\ }f_n\hat{\ }B_n)$. We conclude that $h\in D_i^{\alpha+1}(B)$ and $r_i(B)\geq \alpha+1$.

Assume now that $\alpha<\omega_1$ is a limit ordinal and that $(ii)$ holds for all $\beta<\alpha$. Let $B=\bigcup_n h\hat{\ }f_n\hat{\ }B_n\in\widetilde{\mc B}_{\alpha}$, where $B_n\in\widetilde{\mc B}_{\beta_n}$, $\sup_n \beta_n=\alpha$. For each $\beta <\alpha$, there is some $n\in\omega$ with $\beta_n\geq \beta$. Since $D_i^{\beta_n}(B_n)\neq\emptyset$, we have
\[ D_i^\beta (B)\supset D_i^\beta(h\hat{\ }f_n\hat{\ }B_n)\supset D_i^{\beta_n}(h\hat{\ }f_n\hat{\ }B_n) \supset h\ext f_n\ext D_i^{\beta_n}(B_n) \ni h\hat{\ }f_n\ext s \sqsupset h \]
for some $s\in\omega^{<\omega}$. This implies that $h\in D_i^\beta(B)$ holds for each $\beta<\alpha$, which gives $h\in \underset{\beta<\alpha} \bigcap D_i^\beta(B) = D_i^\alpha(B)$ and $r_i(B)\geq \alpha$.
\end{proof}

From Proposition \ref{proposition: rank of B}, it follows that no $\widetilde {\mc B}$ set can be covered by finitely many sets from $\bigcup_{\beta<\alpha} \mc B_\beta$. We will prove a slightly stronger result later (Lemma \ref{lemma: E, H and finite unions}).

\section{Broom spaces and their \texorpdfstring{$\mc F_\alpha$}{F-alpha} parametrization}\label{section: broom spaces and parametrization}
We will study a certain collection of one-point Lindelöfications of the discrete uncountable space which give the space a special structure. These spaces were used by (among others) Talagrand, who constructed a space which is non-absolutely $\fsd$ (\cite{talagrand1985choquet}). Our goal is to compute the absolute complexity of spaces of this type.

For an arbitrary system $\mc E\subset \mc P \left( \baire \right)$ which contains only countable sets, we define the space $X_{\mc E}$ as the set $\baire\cup\{\infty\}$ with the following topology: each $\sigma\in\baire$ is an isolated point and a neighborhood subbasis of $\infty$ consists of all sets of the form $\{\infty \}\cup (\baire\setminus E)$ for some $E\in\mc E$.

The outline of Section \ref{section: broom spaces and parametrization} is the following: We note in Example \ref{example: Y(T)} that if $X_{\mc E} =: X \subset cX$, then there is a natural formula for an $\fsd$ set $Y_1\subset cX$ containing $X$. Then in Definition \ref{definition: Y_alpha}, we generalize the formula and define sets $Y_\alpha$. These are useful from two reasons: the first is that their definition is `absolute' (it does not depend on the choice of $cX$). The other is that it is easy to compute (an upper bound on) their complexity, which we do in Lemma \ref{lemma: complexity of Y}.

Our goal is to show that $X_{\mc E}=Y_\alpha$ holds for a suitable $\alpha$. In Section \ref{section: auxiliary}, we prepare the tools for establishing a connection between the complexity of $\mc E$ and the value of `the $\alpha$ that works'. In Proposition \ref{proposition: XA are absolutely K alpha}, we use these results to prove that $X_{\mc E}=Y_\alpha$ holds in every compactification. Applying Lemma \ref{lemma: complexity of Y}, we obtain the absolute complexity of $X$. We finish by giving two corollaries of this proposition.

Rather than working with a general $\mc E\subset \mc P(\baire)$, we will be interested in the following systems:

\begin{definition}[Infinite broom sets\footnote{In \cite{talagrand1985choquet} these systems are denoted as $\mc A_\alpha$, but we have to use a different letter to distinguish between infinite brooms and additive Borel classes. The letter $\mc E$ is supposed to stand for ``extension" (of $\mc B$). Note also that Talagrand's collections $\mc A_n$, $n\in\omega$, correspond to our $\mc E_{n+1}$ (for $\alpha\geq \omega$ the enumeration is the same).}]\label{definition: infinite brooms}
For $\alpha\leq\omega_1$ we define
\[ \mc E_\alpha:=\left\{ E\subset \baire | \ E \textrm { is an infinite extension of some } B\in\mc B_\alpha \right\}.\footnote{Where $\mc B_\alpha$ is the $\alpha$-th collection of broom sets, introduced in Definition \ref{definition: brooms}.} \]
\end{definition}
Since $\mc B_{\omega_1}\subset \mc D$, every $E\in \mc E_{\omega_1}$ is closed and discrete in $\baire$, and the following result holds:

\begin{proposition}[\cite{talagrand1985choquet}]
If each $E\in\mc E$ is closed discrete in $\baire$, then $X_{\mc E}\in\fsd\left(\beta X_{\mc E}\right)$.
\end{proposition}

In the remainder of Section \ref{section: broom spaces and parametrization}, $X$ will stand for the space $X_{\mc E}$ for some family $\mc E \subset \mc E_{\omega_1}$ and $cX$ will be a fixed compactification of $X$. All closures will automatically be taken in $cX$. The choice of $\mc E$ is important, because it gives the space $X$ the following property:

\begin{lemma}\label{lemma: x in cX -- X and H}
For every $x\in cX \setminus X$ there exists $H\subset X$ satisfying
\begin{enumerate}[(i)]
\item $H$ can be covered by finite union of elements of $\mc E$;
\item $\left( \forall F\subset X \right) : x\in\overline{F} \implies x\in \overline{F\cap H}$.
\end{enumerate}
\end{lemma}
\begin{proof}
Since $x\neq\infty$, there is some $U\in\mc U(x)$ with $\infty \notin \overline{U}$. Because $U$ is a neighborhood of $x$, we get
\[ \left( \forall F\subset X \right) : x\in\overline{F} \implies x\in\overline{F\cap U}  .\]
Therefore, we define $H$ as $H:=U\cap X$. Clearly, $H$ satisfies the second condition. Moreover, $(i)$ is also satisfied, because by definition of $X_\mc E$, sets of the form $\{\infty\}\cup (\baire \setminus E)$, $E\in\mc E$ form a neighborhood subbasis of $\infty$.
\end{proof}

\subsection{\texorpdfstring{$\mc F$}{F}-Borel sets \texorpdfstring{$Y_\alpha$}{Y-alpha} containing \texorpdfstring{$X_{\mc E}$}{X(E)}}
First, we start with motivation for our definitions:
\begin{example}[$\fsd$ and $F_{\sigma\delta\sigma\delta}$ sets containing $X_{\mc E}$]\label{example: Y(T)}
When trying to compute the complexity of $X$ in $cX$, the following canonical $K_{\sigma\delta}$ candidate comes to mind\footnote{Recall that the notions related to sequences are defined in Notation \ref{notation: sequences}. In particular, $\mc N(s)$ denotes the standard Baire interval.}
\[ X \subset \{\infty\} \cup \underset {n\in\omega} \bigcap \underset {s\in \omega^n} \bigcup \overline{\mc N(s)}=:Y_1. \]
The two sets might not necessarily be identical, but the inclusion above always holds. At the cost of increasing the complexity of the right hand side to $F_{\sigma\delta\sigma\delta}$, we can define a smaller set, which will still contain $X$:
\[ X \subset \{\infty\} \cup \underset {n\in\omega} \bigcap \  \underset {s\in \omega^{n}} \bigcup \ \underset {k\in\omega} \bigcap \ \underset {t\in \omega^k} \bigcup \ \overline{\mc N(s\hat{\ }t)}=:Y_2. \]
\end{example}
In the same manner as in Example \ref{example: Y(T)}, we could define sets $Y_1\supset Y_2 \supset Y_3 \supset \dots \supset X$. Unfortunately, the notation from Example \ref{example: Y(T)} is impractical if we continue any further, and moreover, it is unclear how to extend the definition to infinite ordinals. We solve this problem by introducing `admissible mappings' and a more general definition below.

\begin{definition}[Admissible mappings]\label{definition: admissible mapping}
Let $T\in\textrm{Tr}$ be a tree and $\varphi : T\rightarrow \omega^{<\omega}$ a mapping from $T$ to the space of finite sequences on $\omega$. We will say that $\varphi$ is \emph{admissible}, if it satisfies
\begin{enumerate}[(i)]
\item $(\forall s,t\in T): s \sqsubset t \implies \varphi (s) \sqsubset \varphi(t)$
\item $(\forall t\in T): \left| \varphi (t) \right| = t(0)+t(1)+\dots+t(|t|-1)$.
\end{enumerate}
For $S\subset T$ we denote by $\widetilde{\varphi}(S):=\mathrm{cl}_\mathrm{Tr}(\varphi (S))$ the tree generated by $\varphi (S)$.
\end{definition}

\begin{notation} \label{notation: trees of height alpha}
For each limit ordinal $\alpha<\omega_1$ we fix a bijection $\pi_\alpha : \omega \rightarrow \alpha$. If $\alpha=\beta+1$ is a successor ordinal, we set $\pi_\alpha(n):=\beta$ for each $n\in\omega$. For $\alpha=0$, we define $T_0:=\{\emptyset\}$ to be the tree which only contains the root. For $\alpha\geq 1$, we define $T_\alpha$ (`maximal trees of height $\alpha$') as
\[ T_\alpha := \{\emptyset\} \cup \underset {n\in\omega} \bigcup n\hat{\ }T_{\pi_\alpha (n)}. \]
\end{notation}

\begin{definition} \label{definition: Y_alpha}
For $\alpha<\omega_1$ we define
\[ Y_\alpha:=\left\{ x\in cX | \ \left(\exists \varphi : T_\alpha\rightarrow \omega^{<\omega} \textrm{ adm.} \right)\left( \forall t\in T_\alpha \right) : x\in \overline {\mc N \left( \varphi (t)\right)} \right\}\cup \left\{ \infty \right\} .\]
\end{definition}
For $\alpha=1$, it is not hard to see that this definition coincides with the one given in Example \ref{example: Y(T)}. This follows from the equivalence
\begin{equation*}
\begin{split}
x \in \underset {n\in\omega} \bigcap \underset {s\in \omega^n} \bigcup \overline{\mc N(s)} \iff & (\forall n\in\omega)(\exists s_n\in \omega^{<\omega},\ |s_n|=n): x\in\overline { \mc N(s_n) } \\
\overset{(\star)} \iff & (\exists \varphi : T_1 \rightarrow \omega^{<\omega} \textrm{ adm.})(\forall t\in T_1) :x\in\overline{\mc N\left(\varphi(t)\right)} \\
\overset{\textrm{def}} \iff & x\in Y_1,
\end{split}
\end{equation*}
where $(\star)$ holds because the formula $\varphi(\emptyset):=\emptyset$, $\varphi((n)):=s_n$ defines an admissible mapping on $T_1=\{\emptyset\}\cup\{ (n)| \ n\in\omega \}$ and also any admissible mapping on $T_1$ is defined in this way for some sequences  $s_n$, $n\in\omega$. The definition also coincides with the one given above for $\alpha=2$ -- this follows from the proof of Lemma \ref{lemma: complexity of Y}. We now state the main properties of sets $Y_\alpha$.

\begin{lemma}
$Y_\alpha\supset X$ holds for every $\alpha<\omega_1$.
\end{lemma}
\begin{proof}
Let $\sigma\in \baire$. For $t\in T_\alpha$ we set $\varphi(t):=\sigma|_{t(0)+\dots+t(|t|-1)}$, and observe that this is an admissible mapping which witnesses that $\sigma\in Y_\alpha$.
\end{proof}

\begin{lemma}[Complexity of $Y_\alpha$]\label{lemma: complexity of Y}
For any $\alpha=\lambda+m<\omega_1$ (where $m\in\omega$ and $\lambda$ is either zero or a limit ordinal), $Y_\alpha$ belongs to $\mc F_{\lambda + 2m+1}(cX)$.
\end{lemma}
\begin{proof}
For each $h\in\omega^{<\omega}$ and $\alpha<\omega_1$, we will show that the set $Y^h_\alpha$ defined as
\[ Y^h_\alpha:=\left\{ x\in cX| \ (\exists \varphi : T_\alpha \rightarrow \omega^{<\omega} \textrm{ adm.})(\forall t\in T_\alpha):x\in\overline{\mc N(h\hat{\ }\varphi(t)} \right\} \]
belongs to $\mc F_{\lambda + 2m+1}(cX)$.
The only admissible mapping from $T_0$ is the mapping $\varphi:\emptyset\mapsto \emptyset$, which means that $Y^h_\alpha=\overline{\mc N(h)}\in \mc F_1$. In particular, the claim holds for $\alpha=0$.

Let $1\leq\alpha<\omega_1$. First, we prove the following series of equivalences
\begin{equation*}
\begin{split}
x \in Y^h_\alpha \overset{(a)} \iff & (\exists \varphi : T_\alpha \rightarrow \omega^{<\omega} \textrm{ adm.})(\forall t\in T_\alpha): \\
	& x\in \overline{\mc N \left( h\hat{\ }\varphi(t)\right) } \\
\overset{(b)} \iff & (\exists \varphi : T_\alpha \rightarrow \omega^{<\omega} \textrm{ adm.})(\forall n\in\omega)(\forall t\in T_{\pi_\alpha(n)}): \\
	& x\in\overline{\mc N\left(h\hat{\ }\varphi(n\hat{\ }t)\right)}  \\
\overset{(c)} \iff & (\forall n\in\omega)(\exists s_n\in\omega^ n)(\exists \varphi_n : T_{\pi_\alpha(n)} \rightarrow \omega^{<\omega} \textrm{ adm.})(\forall t\in T_{\pi_\alpha(n)}): \\
	& x\in\overline{\mc N\left(h\hat{\ }s_n\hat{\ }\varphi_n(t)\right)} \\
\overset{(d)} \iff & x\in\bigcap_{n\in\omega}\bigcup_{s_n\in\omega^n}Y^{h\hat{\ }s_n}_{\pi_\alpha(n)},
\end{split}
\end{equation*}
from which the equation \eqref{eq: Y_alpha} follows:
\begin{equation}\label{eq: Y_alpha}
Y^h_\alpha=\bigcap_{n\in\omega}\bigcup_{s_n\in\omega^n}Y^{h\hat{\ }s_n}_{\pi_\alpha(n)}.
\end{equation}
The equivalences $(a)$ and $(d)$ are simply the definition of $Y^h_\alpha$, and $(b)$ follows from the definition of trees $T_\alpha$. The nontrivial part is $(c)$. To prove the implication `$\implies$', observe that if $\varphi$ is admissible, then
\[ (\forall n\in\omega)(\forall t\in T_{\pi_\alpha(n)})(\exists \varphi_n(t)\in \omega^{<\omega}): \varphi(n\ext t)=\varphi\left((n)\right)\ext \varphi_n(t),  \]
the mappings $\varphi_n:T_{\pi_\alpha(n)}\rightarrow \omega^{<\omega}$ defined by this formula are admissible and $|\varphi((n))|=n$. The implication from right to left follows from the fact that whenever $\varphi_n:T_{\pi_\alpha(n)}\rightarrow \omega^{<\omega}$ are admissible mappings and $s_n\in\omega^n$, the mapping $\varphi$ defined by formula $\varphi(\emptyset):=\emptyset$, $\varphi(n\ext t):=s_n\ext \varphi_n(t)$ is admissible as well.

We now finish the proof. If $\alpha$ is a successor ordinal, then $\pi_\alpha(n)=\alpha-1$ holds for all $n\in\omega$. Therefore, we can rewrite \eqref{eq: Y_alpha} as
\begin{equation*}
Y^h_\alpha = \bigcap_{n\in\omega}\bigcup_{s_n\in\omega^n}Y^{h\hat{\ }s_n}_{\pi_\alpha(n)} 
= \bigcap_{n\in\omega}\bigcup_{s_n\in\omega^n}Y^{h\hat{\ }s_n}_{\alpha-1}
\end{equation*}
and observe that each `successor' step increases the complexity by $(\cdot )_{\sigma\delta}$. Lastly, assume that $\alpha$ is a limit ordinal and $Y^s_{\alpha'}\in \underset {\beta<\alpha} \bigcup \mc F_\beta$ holds for all $s\in\omega^{<\omega}$ and $\alpha'<\alpha$. Since for limit $\alpha$, we have $\pi_\alpha(n)<\alpha$ for each $n\in\omega$, \eqref{eq: Y_alpha} gives
\begin{equation*}
Y^h_\alpha = \bigcap_{n\in\omega}\bigcup_{s_n\in\omega^n}Y^{h\hat{\ }s_n}_{\pi_\alpha(n)} 
\in \left( \underset {\beta<\alpha} \bigcup \mc F_\beta  \right)_{\sigma\delta}=\mc F_{\alpha+1},
\end{equation*}
which is what we wanted to prove.
\end{proof}

\subsection{Some auxiliary results} \label{section: auxiliary}
In this section, we give a few tools which will be required to obtain our main result. Later, we will need to show that for a suitable $\alpha$, $Y_\alpha \setminus X$ is empty. In order to do this, we first explore some properties of those $x\in Y_\alpha$ which are in $X$, and of those $x\in Y_\alpha$ which do not belong to $X$. We address these two possibilities separately in Lemma \ref{lemma: Y(T) -- IF case} and Lemma \ref{lemma: Y(T) -- WF case} and show that they are related with the properties of the corresponding admissible mappings\footnote{By such a mapping we mean an admissible mapping which witnesses that $x\in Y_\alpha$ (for details, see Definitions \ref{definition: admissible mapping} and \ref{definition: Y_alpha}).}.
\begin{lemma}[The non-WF case]\label{lemma: Y(T) -- IF case}
For any $\sigma\in\baire$ and any increasing sequence of integers $(n_k)_{k\in\omega}$, we have $\bigcap_k\overline{\mc N(\sigma|n_k)}^{cX}=\{\sigma\}$. In particular, we have
\[ \bigcap_{k\in\omega} \overline{\mc N(\sigma|n_k)}^{cX}\subset X . \]
\end{lemma}
\begin{proof}
Suppose that $\sigma$ and $n_k$ are as above and the intersection $\bigcap_k \overline{\mc N(\sigma|n_k)}$ contains some $x\ne \sigma$. Clearly, we have $x\in cX\setminus X$. By Lemma \ref{lemma: x in cX -- X and H}, there is some discrete (in $\baire$) set $H\subset \baire$, such that for each $k\in\omega$, $x$ belongs to $\overline{\mc N(\sigma|n_k)\cap H}$. Since the sets $\mc N(\sigma|n_k)$ form a neighborhood basis of $\sigma$, there exists $n_0\in\omega$, such that $\mc N(s_{n_0})\cap H$ is a singleton. In particular, this means that $\mc N(s_{n_0})\cap H$ is closed in $cX$, and we have $x\in \overline{\mc N(s_{n_0})\cap H} = \mc N(s_{n_0})\cap H \subset H \subset X$ -- a contradiction.
\end{proof}

On the other hand, Lemma \ref{lemma: Y(T) -- IF case} shows that if the assumptions of Lemma \ref{lemma: Y(T) -- IF case} are not satisfied, we can find a broom set\footnote{\label{footnote: tilde B}Recall that the collections $\widetilde{B}_\alpha$ were introduced in Definition \ref{definition: tilde brooms}.} of class $\alpha$ which corresponds to $x$. We remark that Lemma \ref{lemma: Y(T) -- WF case} is a key step on the way to Proposition \ref{proposition: XA are absolutely K alpha} -- in particular, it is the step in which the admissible mappings are an extremely useful tool.
\begin{lemma}[Well founded case]\label{lemma: Y(T) -- WF case}
Let $x\in Y_\alpha$ and  suppose that $\varphi$ is any admissible mapping witnessing this fact. If $\widetilde{\varphi} (T_\alpha)\in\textrm{WF}$, then there exists $B \in \widetilde{ \mc B} _\alpha$ with $B\subset \widetilde{\varphi} (T_\alpha)$.
\end{lemma}
\begin{proof}
We prove the statement by induction over $\alpha$. For $\alpha=0$ we have $T_0=\{\emptyset\}$ and $\varphi(\emptyset)\in\omega^{<\omega} = \widetilde{\mc B}_0$.

Consider $\alpha > 0$ and assume that the statement holds for every $\beta < \alpha$. We have $\{(n)|\ n\in \omega\} \subset T_{\alpha}$. By the second defining property of admissible mappings, we get $\varphi((m))\neq \varphi((n))$ for distinct $m,n\in\omega$, which means that the tree $T:=\widetilde{\varphi} (\{(n)| \ n\in\omega\})$ is infinite. Since $T \subset \widetilde{\varphi} (T_{\alpha}) \in \textrm{WF}$, we use König's lemma to deduce that $T$ contains some $h\in\omega^{<\omega}$ with infinite set of successors $S:=\textrm{succ}_T(h)$.

For each $s\in S$ we choose one $n_s\in\omega$ with $\varphi(n_s)\sqsupset s$. In the previous paragraph, we have observed that $\left(s(|h|)\right)_{s\in S}$ is a forking sequence. From the first property of admissible mappings, we get
\begin{equation} \label{eq: WF case}
(\forall s\in S)(\forall t \in T_{\alpha} \textrm{ s.t. } t(0)=n_s): \varphi(t)\sqsupset s.
\end{equation}
In particular, for each $s\in S$ we have
\[ \varphi ( \{ t\in T_{\alpha}| \ t(0)=n_s \} ) = s\hat{\ }S_s \textrm{ for some }S_s\subset \omega^{<\omega}. \]
By definition\footnote{Recall that $T_\alpha$ is an $\omega$-ary tree of height $\alpha$, defined in Notation \ref{notation: trees of height alpha}. } of $T_{\alpha}$, we have $\{ t\in \omega^{<\omega}| \ n_s\hat{\ }t\in T_{\alpha} \}=T_{\pi_\alpha (n_s)}$. By \eqref{eq: WF case}, each $\varphi (n_s \hat{\ } t)$ for $t\in T_{\pi_\alpha (n_s)}$ is of the form $\varphi(n_s\hat{\ }t)=\varphi(n_s)\hat{\ }\varphi_s(t)$ for some $\varphi_s(t)\in\omega^{<\omega}$, and, since $\varphi$ is admissible, the mapping $\varphi_s(t):T_{\pi_\alpha (n_s)}\rightarrow\omega^{<\omega}$ defined by this formula is admissible as well.

It follows from the induction hypothesis that $\mathrm{cl}_\mathrm{Tr}\left( \varphi_s (T_{\pi_\alpha (n)})  \right)$ contains some $\widetilde{\mc B}_{\pi_\alpha (n)}$-set $D_s$. Let $C_s \subset \varphi_s (T_{\pi_\alpha (n)})$ be some finite extension of $D_s$. By Remark \ref{remark: basic broom properties}, $C_s$ belongs to $\widetilde{\mc B}_{\pi_\alpha (n)}$. Finally, since
\[ s\ext S_s=\varphi(n_s)\hat{\ }\varphi_s(T_{\pi_\alpha(n_s)})\supset \varphi(n_s)\ext C_s \in \widetilde{\mc B}_{\pi_\alpha (n)} , \]
there also exists some $B_s\subset S_s$ with $B_s\in \widetilde{\mc B}_{\pi_\alpha (n)}$ (the set $B_s$ can have the same bristles and forking sequence as $C_s$, but $s$ might be shorter than $\varphi(n_s)$, so $B_s$ might have a longer handle than $C_s$). We finish the proof by observing that
\begin{equation}\label{eq: Phi contains Balpha}
\widetilde{\varphi} (T_\alpha) \supset \varphi \left( \underset{s\in S} \bigcup \left\{ t\in T_{\alpha}| \ t(0)=n_s \right\} \right) \supset \underset{s\in S} \bigcup s\hat{\ }S_s \supset \underset{s\in S} \bigcup h\hat{\ }s(|h|)\hat{\ }B_s.
\end{equation}
Because the set $S$ is infinite, the rightmost set in \eqref{eq: Phi contains Balpha} is, by definition, an element of $\widetilde{ \mc B} _\alpha$.
\end{proof}

We will eventually want to show that the assumptions of Lemma \ref{lemma: Y(T) -- WF case} can only be satisfied if the family $\mc E$ is sufficiently rich. To this end, we need to extend Lemma \ref{lemma: rank and finite covers} to a situation where a set $H\subset\baire$ cannot be covered by finitely many `infinite broom sets'\footnote{That is, the infinite extensions of broom sets, introduced in Definition \ref{definition: infinite brooms}.}:
\begin{lemma}\label{lemma: E, H and finite unions}
Suppose that a set $H\subset \baire$ contains an infinite extension of some $B\subset \omega^{<\omega}$ which satisfies $r_i(B)\geq \alpha$.\footnote{Recall that $r_i$ denotes the `infinite branching rank' introduced in Definition \ref{definition: derivative}.} Then $H$ cannot be covered by finitely many elements of $\bigcup_{\beta<\alpha} \mc E_\beta$.
\end{lemma}
\begin{proof}
Let $\alpha<\omega_1$, $B$ and $H$ be as above. For contradiction, assume that $H\subset \bigcup_{j=0}^k E_j $ holds for some $E_j\in \mc E_{\beta_j}$, $\beta_j<\alpha$.

By definition of $\mc E_{\beta_j}$, each $E_j$ is an infinite extension of some $B_j\in \mc B_{\beta_j}$ -- in other words, there are bijections $\psi_j:B_j \rightarrow E_j$ satisfying $\psi_j(s)\sqsupset s$ for each $s\in B_j$. Similarly, $H$ contains some infinite extension of $B$, which means there exists an injective mapping $\psi : B\rightarrow H$ satisfying $\psi(s)\sqsupset s$ for each $s\in B$. Clearly, we have
\begin{equation}\label{eq: psi B is covered by psi B_j}
\psi(B)\subset H \subset \bigcup_{j=0}^k E_j = \bigcup_{j=0}^k \psi_j(B_j) .
\end{equation}

Firstly, observe that if $t\in B$ and $s\in B_j$ satisfy $\psi_j(s)=\psi(t)$, the sequences $s$ and $t$ must be comparable. This means that for each $j\leq k$, the following formula correctly defines a mapping $\varphi_j : B_j\rightarrow \omega^{<\omega}$:
$$
\varphi_j(s):=
\begin{cases}
s,  & \textrm{ for } \psi_j(s) \notin \psi(B), \\
\psi^{-1}(\psi_j(s)), & \textrm{ for } s\sqsubset \psi^{-1}(\psi_j(s)), \\
s, & \textrm{ for } s\sqsupset \psi^{-1}(\psi_j(s)),
\end{cases}
$$
which satisfies $\varphi_j(s)\sqsupset s$ for each $s\in B_j$. Since no two elements of $B_j$ are comparable, it also follows that $\varphi_j$ is injective and thus $\varphi(B_j)$ is an extension of $B_j$. By Remark \ref{remark: basic broom properties}, this implies that $\varphi_j(B_j)\in\mc B_{\beta_j}$.

By \eqref{eq: psi B is covered by psi B_j}, for $t\in B$ there exists some $j\leq k$ and $s\in B_j$, such that $\psi_j(s)=\psi(t)$. It follows from the definition of $\psi_j$ that
\[ t=\psi^{-1}(\psi(t))=\psi^{-1}(\psi_j(s))\sqsubset \varphi_j(s) .\] In other words, we get
\[ B\subset \mathrm{cl}_\mathrm{Tr}\left( \bigcup_{j=0}^k \varphi_j(B_j) \right) .\]
We observe that this leads to a contradiction:
\begin{align*}
\alpha & = r_i(B) \leq r_i \left( \mathrm{cl}_\mathrm{Tr}\left( \bigcup_{j=0}^k \varphi_j(B_j) \right) \right) = r_i\left( \bigcup_{j=0}^k \varphi_j(B_j) \right) = \\
	& = \max_j \ r_i \left( \varphi_j(B_j) \right) \leq \max_j \beta_j < \alpha.
\end{align*}
\end{proof}

\subsection{Absolute complexity of \texorpdfstring{$X_{\mc E}$}{X(E)}} \label{section: absolute complexity of brooms}
In this section, we prove an upper bound on the absolute complexity of spaces $X_{\mc E}$ for any collection of infinite broom sets $\mc E$. Talagrand's earlier result will then imply that, in some cases, this bound is sharp.
\begin{proposition}[Complexity of $X_{\mc E}$]\label{proposition: XA are absolutely K alpha}
For any integer $m\in\omega$ and limit ordinal $\lambda < \omega_1$, we have
\begin{enumerate}[1)]
\item $\mc E\subset \mc E_m$ $\implies$ $X_{\mc E}$ is absolutely $\mc F_{2m+1}$;
\item $\mc E\subset \underset {\beta<\lambda} \bigcup \mc E_\beta$ $\implies$ $X_{\mc E}$ is absolutely $\mc F_{\lambda+1}$;
\item $\mc E\subset \mc E_{\lambda+m}$ $\implies$ $X_{\mc E}$ is absolutely $\mc F_{\lambda+2m+3}$.
\end{enumerate}
\end{proposition}
\begin{proof}
Suppose that $\lambda$ and $m$ are as above, $\mc E$ is a family of discrete subsets of $\baire$ and $cX$ is a compactification of the space $X:=X_{\mc E}$. Denote $\alpha_1 := m$, $\alpha_2:=\lambda$ and $\alpha_3:=\lambda+m+1$. Note that we are only going to use the following property of $\mc E$ and $\alpha_i$:
\begin{equation} \label{eq: main proposition cases}
\begin{split}
 & \textrm{$\mc E$ satisfies the hypothesis of } i) \textrm{ for } i=1 \ \ \ \ \ \implies \mc E\subset \bigcup_{\beta<{\alpha_i+1}} \mc E_\beta, \\
 & \textrm{$\mc E$ satisfies the hypothesis of } i) \textrm{ for } i\in\{2,3\} \implies \mc E\subset \bigcup_{\beta<{\alpha_i}} \mc E_\beta.
\end{split}
\end{equation}
We will show that $X=Y_{\alpha_i}$ holds in each of these cases\footnote{Where $Y_\alpha$ is given by Definition \ref{definition: Y_alpha}.}. Once we have this identity, the conclusion immediately follows from Lemma \ref{lemma: complexity of Y}.

Let $i\in\{1,2,3\}$. Suppose that there exists $x\in Y_{\alpha_i}\setminus X$ and let $\varphi : T_{\alpha_i} \rightarrow \omega^{<\omega}$ be an admissible mapping\footnote{\label{footnote: tilde admissible mapping}Recall that admissible mappings were introduced in Definition \ref{definition: admissible mapping}. The `tilde version' is defined as $\widetilde{\varphi}(S):=\mathrm{cl}_\mathrm{Tr}(\varphi (S))$. } witnessing that $x\in Y_{\alpha_i}$. Moreover, let $H\subset \baire$ be a set satisfying the conclusion of Lemma \ref{lemma: x in cX -- X and H}.

Recall that by definition of $Y_{\alpha_i}$, we have
\begin{equation} \label{eq: def of Y alpha}
\left(\forall t\in T_{\alpha_i}\right): x\in\ \overline{ \mc N(\varphi(t))}.
\end{equation}
By Lemma \ref{lemma: Y(T) -- IF case}, $\widetilde{\varphi} (T_{\alpha_i})$ contains no infinite branches. Therefore, we can apply Lemma \ref{lemma: Y(T) -- WF case} to obtain a $\widetilde{\mc B}_{\alpha_i}$-set\footnote{$\widetilde {\mc B}_\alpha$ is the `disjoint version' of broom collection $\mc B_\alpha$ (see Definitions \ref{definition: brooms} and \ref{definition: tilde brooms}).} $B\subset \widetilde{\varphi}( T_{\alpha_i})$. Since, by \eqref{eq: def of Y alpha}, $x$ belongs to $\overline{ \mc N(s)}$ for each $s\in B$, we conclude that $x\in\overline{H\cap \mc N(s)}$ holds for every $s\in B$ (Lemma \ref{lemma: x in cX -- X and H}, $(ii)$). In particular, all the intersections $H\cap \mc N(s)$ must be infinite. Note that for the next part of the proof, we will only need the intersections to be non-empty. The fact that they are infinite will be used in the last part of the proof.

We can now conclude the proof of $2)$ and $3)$. Let $i\in\{2,3\}$ and assume that the hypothesis of $i)$ holds. Since $H\cap \mc N(s)$ is non-empty for each $s\in B$, it follows that $H$ contains an infinite extension of $B$. Since $r_i(B)=\alpha_i$ holds by Lemma \ref{proposition: rank of B}, Lemma \ref{lemma: E, H and finite unions} yields that $H$ cannot be covered by finitely many elements of $\bigcup_{\beta<{\alpha_i}} \mc E_\beta$. By \eqref{eq: main proposition cases}, we have $\mc E\subset \bigcup_{\beta<{\alpha_i}} \mc E_\beta$. This contradicts the first part of Lemma \ref{lemma: x in cX -- X and H} (which claims that $H$ \emph{can} be covered by finitely many elements of $\mc E$).

For the conclusion of the proof of $1)$, let $i=1$ and assume by \eqref{eq: main proposition cases} that $\mc E\subset  \bigcup_{\beta<\alpha_i+1} \mc E_\beta$ holds. Compare this situation with the setting from the previous paragraph -- it is clear that if we show that $H$ in fact contains an infinite extension of some $\widetilde B$ with $r_i(\widetilde B)\geq \alpha_i+1$, we can replace $B$ by $\widetilde B$. We can then apply the same proof which worked for $2)$ and $3)$.

To find $\widetilde B$, enumerate $B$ as $B=\{ s_n | \ n\in\omega \} $. Since each $H\cap \mc N(s_n)$ is infinite, there exist distinct sequences $\sigma^k_n\in H$, $k\in\omega$, satisfying $s_n\sqsubset\sigma^k_n$. We use these sequences to obtain $s_n^k\in\omega^{<\omega}$ for $n,k\in\omega$, which satisfy $s_n\sqsubset s^k_n\sqsubset \sigma^k_n$ and $k\neq l$ $\implies$ $s^k_n\neq s^l_n$. We denote $\widetilde{B}:=\{ s^k_n | \ n,k\in\omega \}$. Clearly, this set satisfies $D_i(\widetilde B)\supset B$ and $H$ contains some infinite extension of $\widetilde B$. Because $r_i(B)=\alpha_i$ is finite, we have $r_i(\widetilde{B}) \geq 1+\alpha_i \overset{\alpha_i<\omega}{=} \alpha_i +1 > \alpha_i = r_i(B)$, which completes the proof.
\end{proof}

Applying Proposition \ref{proposition: XA are absolutely K alpha} to the construction from \cite{talagrand1985choquet} immediately yields the following corollary:
\begin{theorem}\label{theorem: main theorem}
For every even $4\leq \alpha < \omega_1$, there exists a space $T_\alpha$ with properties
\begin{enumerate}[(i)]
\item $T_\alpha$ is an $F_{\sigma\delta}$ space;
\item $T_\alpha$ is not an absolute $\mc F_\alpha$ space;
\item $T_\alpha$ is an absolute $\mc F_{\alpha+1}$ space.
\end{enumerate}
\end{theorem}
\begin{proof}
The existence of spaces $T_\alpha$ which satisfy $(i)$ and $(ii)$ follows from \cite{talagrand1985choquet}. In this paper, Talagrand constructed a family $\mc A_T\subset \mc E_{\omega_1}$, which is both `rich enough' and almost-disjoint (that is, the intersection of any two families is finite). This is accomplished by a smart choice of bristles and forking sequences on the `highest level of each broom'.

He then shows that for a suitable $\widetilde \alpha$, any space $X_{\mc E}$ with $\mc E\supset \mc A_T \cap \mc E_{\widetilde \alpha}$, where $\mc E$ is almost-disjoint, satisfies $(i)$ and $(ii)$. In particular, this holds for $T_\alpha := X_{\mc A_T \cap \mc E_{\widetilde \alpha}}$. Note however that Talagrand was interested in the `maximal' version of the construction, which is the space $T:= X_{\mc A_T}$. We, on the other hand, will use the `intermediate steps' of his construction.

For us, the details of the construction are not relevant -- the only properties we need are $(i)$, $(ii)$ and the correspondence between $\alpha$ and $\widetilde \alpha$.\footnote{Note that in \cite{talagrand1985choquet}, the author uses slightly different (but equivalent) definition of broom families $\mc E_\alpha$, which shifts their numbering for finite $\alpha$ by $1$.} This correspondence is such that the broom class only increases with countable intersections (that is, odd steps of $\mc F_\alpha$), which translates to our notation as follows:
\begin{itemize}
\item $\alpha=2n\in\omega$, $\alpha\geq 4$ $\implies T_{\alpha}=X_{\mc E}$ for some ${\mc E}\subset \mc E_n$;
\item $\alpha<\omega_1$ is a limit ordinal $\implies T_\alpha=X_{\mc E}$ for some ${\mc E}\subset \underset{\beta<\alpha} \bigcup \mc E_\beta$;
\item $\alpha=\lambda+2n+2$ for $n\in\omega$ and limit $\lambda$ $\implies T_\alpha=X_{\mc E}$ for some ${\mc E}\subset \mc E_{\lambda+n}$.
\end{itemize}
In all the cases, it follows from Proposition \ref{proposition: XA are absolutely K alpha}  that $X_{\mc E}$ is absolutely $\mc F_{\alpha+1}$.
\end{proof}

If $\beta<\omega_1$ is the least ordinal for which some $X$ is an (absolute) $\mc F_\beta$ space, we say that the \emph{(absolute) complexity} of $X$ is $\mc F_\beta$. Using this notation, Theorem \ref{theorem: main theorem} can be rephrased as ``for every odd $5\leq \alpha<\omega_1$, there exists an $\fsd$ space whose absolute complexity is $\mc F_\alpha$''. By modifying the spaces from Theorem \ref{theorem: main theorem}, we obtain the following result:

\begin{corollary}\label{corollary: non K alpha space which is abs K alpha+1}
For every two countable ordinals $\alpha\geq \beta \geq 3$, $\alpha$ odd, there exists a space $X_{\alpha,\beta}$, such that
\begin{enumerate}[(i)]
\item the complexity of $X_{\alpha,\beta}$ is $\mc F_\beta$;
\item the absolute complexity of $X_{\alpha,\beta}$ is $\mc F_\alpha$.
\end{enumerate}
\end{corollary}
\begin{proof}
If $\alpha=\beta$, this is an immediate consequence of absoluteness of Borel classes. To make this claim more precise, let $P$ be some uncountable Polish space. By Remark \ref{remark: F,G and M,A} $(ii)$, $\mc F_\beta(P)$ corresponds to some Borel class $\mc C(P)$. We define $Z_\beta$ as one of the subspaces of $P$, which are of the class $\mc C(P)$, but not of the `dual' class -- that is, if $\mc C(P)=\Sigma^0_\gamma(P)$ for some $\gamma<\omega_1$, then $Z_\beta\in \Sigma^0_\gamma(P)\setminus \Pi^0_\gamma (P)$ (and vice versa for $\mc C(P)=\Pi^0_\gamma(P)$). For the existence of such a set, see for example Corollary 3.6.8 in \cite{srivastava2008course}.

Since $Z_\beta$ is separable and metrizable, Theorem \ref{theorem: separable metrizable} guarantees that it is an absolute $\mc F_\beta$ space. However, it is not of the class $\mc F_{\beta'}$ for any $\beta'<\beta$, because that would by Remark \ref{remark: F,G and M,A} imply that $Z_\beta$ belongs to $\mc F_{\beta'} (P)\subset \Pi^0_\gamma(P)$. This shows that if $\beta=\alpha$, the space $X_{\alpha,\beta}:=Z_\beta$ has the desired properties.

If $\beta<\alpha$, we define $X$ as the topological sum of $Z_\beta$ and the space $T_{\alpha-1}$ from Theorem \ref{theorem: main theorem}. Since $T_{\alpha-1}\in \fsd(\beta T_{\alpha-1})=\mc F_3(\beta T_{\alpha-1})$ and $\beta \geq 3$, we get that $X_{\alpha,\beta}$ is of the class $\mc F_\beta$ in the topological sum  $\beta X_{\alpha,\beta} = \beta Z_\beta \oplus \beta T_{\alpha-1}$. By previous paragraph, it is of no lower class in $\beta X_{\alpha,\beta}$. By Theorem \ref{theorem: main theorem} (and the previous paragraph), $X_{\alpha,\beta}$ also has the correct absolute complexity.
\end{proof}

To get a complete picture of possible combinations of complexity and absolute complexity for $\mc F$-Borel spaces, it remains to answer the following questions:
\begin{problem}\label{problem: X alpha beta}
\begin{enumerate}[(i)]
\item Let $4\leq\alpha<\omega_1$ be an even ordinal number. Does there exists a space $X_\alpha$, whose absolute complexity is $\mc F_\alpha$ and complexity is $\mc F_\beta$ for some $\beta<\alpha$? What is the lowest possible value of $\beta$?
\item If a space $X$ is $\mc F$-Borel in every compactification, is it necessarily absolutely $\mc F_\alpha$ for some $\alpha<\omega_1$?
\end{enumerate}
\end{problem}
One could expect that Corollary \ref{corollary: non K alpha space which is abs K alpha+1} might also hold for even $\alpha$, so that we should be able to find an $\fsd$ space answering the first part of Problem \ref{problem: X alpha beta} in positive.

\section*{Acknowledgment}
I would like to thank my supervisor, Ondřej Kalenda, for numerous very helpful
suggestions and fruitful consultations regarding this paper. This work was supported by the research grant GAUK No. 915.


\end{document}